\documentclass[11pt,a4paper,reqno]{article}

\usepackage[utf8]{inputenc}
\usepackage[affil-it]{authblk}
\usepackage{amsfonts, amsmath, amssymb, amsthm, bbm, color, enumerate, graphicx, mathtools, tikz, hyperref, enumerate}
\usepackage[numbers]{natbib}
\usepackage[T1]{fontenc}
\usepackage[a4paper,includeheadfoot, total={6in,8.8in}]{geometry}

\long\def\/*#1*/{}

\newcommand{\prob}{\ensuremath{\xrightarrow{\mathbbm{P}}}} 
\newcommand{\dist}{\ensuremath{\xrightarrow{\mathcal{L}}}} 

\newtheorem{theorem}{Theorem}

\newtheorem{lemma}[theorem]{Lemma}
\newtheorem{proposition}[theorem]{Proposition}
\newtheorem{remark}[theorem]{Remark}

\hypersetup{
  colorlinks,
  linkcolor=blue,          
  citecolor=red,       
  filecolor=blue,   
  urlcolor=black,
  pdftitle={},
  pdfauthor={D. Mukherjee},
  pdfcreator={D. Mukherjee},
  pdfsubject={},
  pdfkeywords={}
}

\begin{document}

\title{Universality of Load Balancing Schemes on Diffusion Scale}

\author[1]{D. Mukherjee\footnote{Corresponding author: \texttt{d.mukherjee@tue.nl}}}
\author[1,2]{S.C. Borst}
\author[1]{J.S.H. van Leeuwaarden}
\author[3]{P.A. Whiting}
\affil[1]{Department of Mathematics and Computer Science,
Eindhoven University of Technology, The Netherlands}
\affil[2]{Alcatel-Lucent Bell Labs, Murray Hill, NJ, USA}
\affil[3]{Department of Engineering,
Macquarie University, North Ryde, NSW, Australia}

\renewcommand\Authands{ and }

\date{\today}

\maketitle 

\begin{abstract}
We consider a system of $N$~parallel queues with identical exponential service rates and a single dispatcher where tasks arrive as a Poisson process. When a task arrives, the dispatcher always assigns it to an idle server, if there is any, and to a server with the shortest queue among $d$~randomly selected servers otherwise ($1 \leq d \leq N$). This load balancing scheme subsumes the so-called Join-the-Idle Queue (JIQ) policy ($d = 1$) and the celebrated Join-the-Shortest Queue (JSQ) policy ($d = N$) as two crucial special cases.
We develop a stochastic coupling construction to obtain the diffusion limit of the queue process in the Halfin-Whitt heavy-traffic regime, and establish that it does not depend on the value of~$d$, implying that assigning tasks to idle servers is sufficient for diffusion level optimality.
\end{abstract}



\section{Introduction}

In the present paper we establish a universality property for a broad class of load balancing schemes in a many-server heavy-traffic regime. While the specific features of load balancing policies may considerably differ, the principal purpose is to distribute service requests or tasks among servers or geographically distributed nodes in parallel-processing systems. Well-designed load balancing schemes provide an effective mechanism for improving relevant performance metrics experienced by users while achieving high resource utilization levels.
The analysis and design of load balancing policies has attracted strong renewed interest in the last several years, mainly motivated by significant challenges involved in assigning tasks (e.g.~file transfers, compute jobs, data base look-ups) to servers in large-scale data centers.

Load balancing schemes can be broadly categorized as static (open-loop), dynamic (closed-loop), or some intermediate blend, depending on the amount of real-time feedback or state information (e.g.~queue lengths or load measurements) that is used in assigning tasks.
Within the category of dynamic policies, one can further distinguish between push-based and pull-based approaches, depending on whether the initiative resides with a dispatcher actively collecting feedback from the servers, or with the servers advertizing their availability or load status. The use of state information naturally allows dynamic policies to achieve better performance and greater resource pooling gains, but also involves higher implementation complexity and potentially substantial communication overhead. The latter issue is particularly pertinent in large-scale data centers, which deploy thousands of servers and handle massive demands, with service requests coming in at huge rates. 

In the present paper we focus on a basic scenario of $N$~parallel queues with identical servers, exponentially distributed service requirements, and a service discipline at each individual server that is oblivious to the actual service requirements (e.g.~FCFS).
In this canonical case, the so-called Join-the-Shortest-Queue (JSQ) policy has 
several strong optimality properties, and in particular minimizes the overall mean delay among the class of non-anticipating load balancing policies that do not have any advance knowledge of the service requirements \cite{EVW80,W78,Winston77}.
(Relaxing any of the three above-mentioned assumptions tends to break the optimality properties of the JSQ policy, and renders the delay-minimizing policy quite complex or even counter-intuitive,
see for instance \cite{GHSW07,Jonckheere06,Whitt86}.)

In order to implement the JSQ policy, a dispatcher requires instantaneous knowledge of the queue lengths at all the servers, which may give rise to a substantial communication burden, and may not be scalable in scenarios with large numbers of servers.
The latter issue has motivated consideration of so-called JSQ($d$) policies, where the dispatcher assigns an incoming task to a server with the shortest queue among $d$~servers selected uniformly at random.
Mean-field limit theorems in Mitzenmacher~\cite{Mitzenmacher01}
and Vvedenskaya {\em et al.}~\cite{VDK96} indicate that even a value as small as $d = 2$ yields significant performance improvements in a many-server regime, in the sense that the tail of the queue length distribution at each indvidual server falls off much more rapidly compared to a strictly random assignment policy ($d = 1$). This is commonly referred to as the ``power-of-two'' effect. While these results were originally proved for exponential service requirement distributions, they have recently been extended to general service requirement distributions in Bramson {\em et al.}~\cite{BLP12}.

In the present paper we consider a related but different family of load balancing schemes termed JIQ($d$), where the dispatcher always assigns an incoming task to an idle server, if there is any, and to a server with the shortest queue among $d$~uniformly at random selected servers otherwise. Observe that the JIQ($N$) scheme coincides with the ordinary JSQ policy, while the JIQ($1$) scheme corresponds to the so-called Join-the-Idle-Queue
(JIQ) policy considered in~\cite{BB08,LXKGLG11,S15}.
The latter policy offers particularly attractive properties, both from a scalability perspective and from a performance viewpoint. Since only knowledge of the empty queues is required, it suffices for servers to send an `invite' notice to the dispatcher whenever they become idle.  This generates at most one message per task and ensures low communication overhead even in large-scale systems with many servers. At the same time, fluid-limit theorems in Stolyar~\cite{S15} indicate that, for any fixed subcritical load per server, the equilibrium probability of a task experiencing a wait because no idle server is available, asymptotically vanishes in a regime where the number of servers grows large.

We consider a regime where the number of servers~$N$ grows large, but additionally assume that the capacity slack per server diminishes as $\beta / \sqrt{N}$, \emph{i.e.}, the load per server approaches unity as $1 - \beta / \sqrt{N}$, with $\beta > 0$ some positive coefficient. In terms of the aggregate traffic load and total service capacity, this scaling corresponds to the so-called Halfin-Whitt heavy-traffic regime which was introduced in the seminal paper~\cite{HW81} and has been extensively studied since. The set-up in~\cite{HW81}, as well as the numerous model extensions in the literature, predominantly concerned a setting with a single centralized queue and server pool, rather than a scenario with parallel queues. To the best of our knowledge, the only exception is a recent study of Eschenfeldt \& Gamarnik~\cite{EG15}, which considered a parallel-server system with the ordinary JSQ policy, and showed that in the Halfin-Whitt regime the diffusion-scaled system occupancy state weakly converges to a two-dimensional reflected Ornstein-Uhlenbeck process.

In the present paper we exploit a stochastic coupling construction to extend the latter result to the entire class of JIQ($d$) policies. We specifically establish that the diffusion limit, rather surprisingly, does not depend on the value of~$d$ at all, so that in particular the JIQ and JSQ policies yield the same diffusion limit. The latter property implies that in a many-server heavy-traffic regime, ensuring that tasks are assigned to idle servers whenever possible, e.g.~using a low-overhead invite mechanism, suffices to achieve optimality at the diffusion level, and not just at the fluid level as proved in Stolyar~\cite{S15} for the under-loaded scenario. It further suggests that using any additional queue length information beyond the knowledge of empty queues yields only limited performance gains in large-scale systems in the Halfin-Whitt heavy-traffic regime.

The remainder of the paper is organized as follows.
In Section~\ref{sec: model descr} we present a detailed model description and formulate the main result.
In Section~\ref{sec: coupling} we develop a stochastic coupling
construction to compare the system occupancy state under various task assignment policies.
We then combine in Section~\ref{sec: conv} the stochastic comparison results with some of the derivations in~\cite{EG15} to obtain the common diffusion limit and finally make a few concluding remarks in Section~\ref{sec:conclusion}.

\section{Model description}
\label{sec: model descr}
Consider a system with $N$~parallel queues with independent and identical servers having unit-exponential service rates and a single \emph{dispatcher}.
Tasks arrive at the dispatcher as a Poisson process of rate $\lambda_N$, and are instantaneously forwarded to one of the servers. Tasks can be queued at the various servers, possibly subject to a buffer capacity limit as further described below, but \emph{cannot} be queued at the dispatcher. The dispatcher always assigns an incoming task to an idle server, if there is any, and to a server with the shortest queue among $d$~uniformly at random selected servers otherwise ($1 \leq d \leq N$), ties being broken arbitrarily. The buffer capacity at each of the servers is~$b\geq 2$ (possibly infinite), and when a task is assigned to a server with $b$~pending tasks, it is instantly discarded.

As mentioned earlier, the above-described scheme coincides with
the ordinary Join-the-Shortest-Queue (JSQ) policy when $d = N$,
and corresponds to the so-called Join-the-Idle-Queue (JIQ) policy
considered in~\cite{BB08,LXKGLG11,S15} when $d = 1$.

 Under the JSQ policy, the dispatcher always assigns an incoming
task to the server with the minimum queue length.
As stated in the introduction, the JSQ policy has several strong
optimality properties in the symmetric Markovian scenario under
consideration.
In order to implement the JSQ policy however, a dispatcher requires
instantaneous knowledge of the queue lengths at all the servers,
which may give rise to a substantial communication burden,
and may not be scalable in scenarios with large numbers of servers.
In a recent study Eschenfeldt \& Gamarnik~\cite{EG15} characterized
the diffusion limit of the system occupancy state in the Halfin-Whitt
heavy-traffic regime.

 Under the JIQ policy, the dispatcher assigns an incoming task
to an idle server, if there is any, or to a uniformly at random
selected server otherwise.
This scheme is of particular interest because of its low
communication overhead, and can be implemented as follows.
When a server becomes idle, it sends an \textit{invite} message
to the dispatcher declaring that it is vacant.
Whenever a task arrives, the dispatcher looks at its list of invite
messages.
If there are any messages in the list, then it selects one arbitrarily,
assigns the task to the corresponding server, and discards
the selected invite message.
Otherwise the dispatcher assigns the task uniformly at random
to one of the servers.
In this way the number of messages exchanged per task is at most~$1$,
reducing communication overhead and ensuring scalability.
Stolyar~\cite{S15} recently proved that the probability that there
are invite messages approaches one, and hence the fraction of tasks
that incur a non-zero wait tends to zero, in a fluid regime where
the number of servers and total arrival rate grow large in proportion with $\lambda_N/N\to\lambda<1$ as $N\to\infty$.

In the present paper we consider the Halfin-Whitt heavy-traffic regime where the arrival rate increases with the number of servers as $\lambda_N = N-\beta\sqrt{N}$ for some $\beta>0$. 
We denote the class of above-described policies by $\Pi^{(N)}(d)$, where the superscript~$N$ indicates that the diversity parameter~$d$ is allowed to depend on the number of servers. 
For any policy $\Pi \in \Pi^{(N)}(d)$ and buffer size~$b$, let $\mathbf{Q}^\Pi = (Q_1^\Pi, Q_2^\Pi, \ldots, Q_b^\Pi)$, where $Q_i^\Pi$ is the number of servers with a queue length greater than or equal to $i = 1, \ldots, b$, including the possible task in service. 
Also, let $\mathbf{X}^\Pi = (X_1^\Pi, X_2^\Pi, \ldots, X_b^\Pi)$ be a properly centered and scaled version of the vector $\mathbf{Q}^{\Pi}$, with $X_1^\Pi = (Q_1^\Pi-N)/\sqrt{N}$ and $X_i^\Pi = Q_i^\Pi/\sqrt{N}$ for $i = 2, \dots, b$. 
The reason why $Q_1^\Pi$ is centered around~$N$ while $Q_i^\Pi$, $i = 2, \dots, b$, are not, is because the fraction of servers with exactly one task tends to one as $N$ grows large as we will see. 
In case of a finite buffer size $b < \infty$, when a task is discarded, we call it an \emph{overflow} event, and we denote by $L^\Pi(t)$ the total number of overflow events under policy~$\Pi$ up to time~$t$. 
\\

The next theorem states our main result. In the rest of the paper let $D$ be the set of all right continuous functions from $[0,\infty)$ to $\mathbbm{R}$ having left limits and let `$\dist$' denote  convergence in distribution. 

\begin{theorem}
\label{th: main}

For any policy $\Pi \in \Pi^{(N)}(d)$,
if for $i=1,2,\ldots$, $X_i^\Pi(0) \dist X_i(0)$ in $\mathbbm{R}$
as $N \to \infty$ with $X_i(0)=0$ for $i\geq 3$, then the processes
$\{X_i^\Pi(t)\}_{t \geq 0} \dist \{X_i(t)\}_{t \geq 0}$
in~$D$, where $X_i(t) \equiv 0$ for $i \geq 3$ and $(X_1(t), X_2(t))$
are unique solutions in $D \times D$ of the stochastic
integral equations
\begin{equation}\label{eq: main theorem}
\begin{split}
X_1(t) &=
X_1(0) + \sqrt{2} W(t) - \beta t + \int_0^t (-X_1(s)+X_2(s)) ds - U_1(t), \\
X_2(t) &= X_2(0) + U_1(t) + \int_0^t(-X_2(s)) ds,
\end{split}
\end{equation}
where $W$ is a standard Brownian motion and $U_1$ is
the unique non-decreasing non-negative process in~$D$ satisfying
$\int_0^\infty \mathbbm{1}_{[X_1(t)<0]} dU_1(t) = 0$.
\end{theorem}

The above result is proved in~\cite{EG15} for the ordinary JSQ
policy.
Our contribution is to develop a stochastic ordering construction
and establish that, somewhat remarkably, the diffusion limit is
the same for any policy in $\Pi^{(N)}(d)$.
In particular, the JIQ and JSQ policies yield the same diffusion limit.
The latter property implies that in the Halfin-Whitt heavy-traffic regime,
assigning tasks to idle servers, e.g.~through a light-weight invite
mechanism, suffices to achieve optimality at the diffusion level.
It further suggests that using any additional queue length information
beyond the knowledge of empty queues yields only limited performance
gains in large-scale systems in the Halfin-Whitt heavy-traffic regime.

\begin{remark}\textnormal{
We note that as in~\cite{EG15} we assume the convergence of the
initial state, which implies that the process has to start from
a state in which the number of vacant servers as well as the number
of servers with two tasks scale with $\sqrt{N}$,
and the number of servers with three or more tasks is $o(\sqrt{N})$.}
\end{remark}

\section{Coupling and stochastic ordering}
\label{sec: coupling}

In this section we prove several stochastic comparison results for the
system occupancy state under various load balancing schemes for a fixed
number of queues~$N$ (and hence we shall often omit the superscript~$N$
in this section).
These stochastic ordering results will be leveraged in the next section
to prove the main result stated in Theorem~\ref{th: main}.

In order to bring out the full strength of the stochastic comparison results,
we will in fact consider a broader class of load balancing schemes
$\Pi^{(N)} := \{\Pi(d_0, d_1, \ldots, d_{b-1}): d_0 = N, 1 \leq d_i \leq N,
1 \leq i \leq b-1, b \geq 2\}$, and show that Theorem~\ref{th: main} actually
holds for this entire class.
In the scheme $\Pi(d_0, d_1, \ldots, d_{b-1})$, the dispatcher assigns
an incoming task to the server with the minimum queue length among $d_k$
(possibly function of~$N$) servers selected uniformly at random when the
minimum queue length across the system is~$k$, $k = 0, 1, \ldots, b - 1$.
As before, $b$ represents the buffer size, and when a task is assigned
to a server with $b$~outstanding tasks, it is instantly discarded.

\subsection{Stack formation and deterministic ordering}
\label{subsec: det_ord}

Let us consider the servers arranged in non-decreasing order of their queue lengths. Each server along with its queue can be thought of as a stack of items. The ensemble of stacks then represent the empirical CDF of the queue length distribution, and the $i^{th}$ horizontal bar corresponds to $Q_i^{\Pi}$ (for the concerned policy $\Pi$). The items are added to and removed from the various stacks according to some rule. Before proceeding to the coupling argument, we first state and prove a deterministic comparison result under the above setting.

Consider two ensembles $A$ and $B$ with the same total number of stacks. The stacks in ensemble $A$ have a maximum capacity of $b$ items and those in ensemble $B$ have a maximum capacity of $b'$ items with $b\leq b'$. For two such ensembles a step is said to follow $Rule(k,l,l_A,l_B)$ if either addition or removal of an item in both ensembles is done in that step as follows:
\begin{enumerate}[(i)]
\item Removal: An item is removed (if any) from the $k^{th}$   stack from both ensembles or an item is removed from some stack in ensemble $A$ but no removal is done in ensemble $B$.
\item Addition: 
\begin{itemize}
\item[(ii.a)] System A: If the minimum stack height is less than $b-1$, then the item is added to the $l^{th}$  stack. Else, the item is added to the $l_A^{th}$   stack. If the item lands on a stack with height $b$, then it is dropped.
\item[(ii.b)] System B: If the minimum stack height is less than $b-1$, then the item is added to the $l^{th}$  stack. Otherwise if the minimum stack height is precisely equal to $b-1$, the item is added to the $l_B^{th}$  stack. When the minimum stack height in the system is at least $b$, the item can be sent to any stack. If the item lands on a stack with height $b'$, then it is dropped.
\end{itemize}
\end{enumerate}
Then we have the following result.
\begin{proposition}\label{prop: det_ord}
Consider two ensembles $A$ and $B$ as described above with the total number of stacks being $N$, stack capacities being $b$ and $b'$ respectively, $b\leq b'$ and with $\mathbf{Q}^A\leq \mathbf{Q}^B$ component-wise \emph{i.e,} $Q^A_i\leq Q^B_i$ for all $i\geq 1$. The component-wise ordering is preserved if at any step $Rule(k,l,l_A,l_B)$ is followed with $l_A\geq l_B$ and either $l=1$ or $l\geq l_B$.
\end{proposition}
Before diving deeper into the proof of this proposition, let us discuss the high-level intuition behind it.
First observe that, if $\mathbf{Q}^A\leq \mathbf{Q}^B$, and an item is added (removed) to (from) the stack with the same index in both ensembles, then the component-wise ordering will be preserved. Hence, the preservation of ordering at the time of removal, and at the time of addition when, in both ensembles, the minimum stack height is less than $b-1$, is fairly straightforward.

Now, in other cases of addition, since in ensemble $A$ the stack capacity is $b\ (\leq b')$, if the minimum stack height in ensemble $B$ is at least $b$,  the ordering is preserved trivially. This leaves us with only the case when the minimum stack height in ensemble $B$ is precisely equal to $b-1$. In this case, when the minimum stack height in ensemble $A$ is also precisely equal to $b-1$, the preservation of the ordering follows from the assumption that $l_A\geq l_B$, which ensures that if in ensemble $A$, the item is added to some stack with $b-1$ items (and hence increases $Q^A_{b}$), then the same will be done in ensemble $B$ whenever $Q^A_b=Q^B_b$. Otherwise if the minimum stack height in ensemble $A$ is less than $b$, then assuming either $l=1$ (\emph{i.e.}~the item will be sent to the minimum queue) or $l\geq l_B$ (\emph{i.e.}~an increase in $Q^A_b$ implies an increase in $Q_b^B$) ensures the preservation of ordering.
\begin{proof}[Proof of Proposition~\ref{prop: det_ord}]
Suppose after following $Rule(k,l,l_A,l_B)$ the updated stack heights of ensemble $\Pi$ are denoted by $(\tilde{Q}_1^{\Pi},\tilde{Q}_2^{\Pi},\ldots)$, $\Pi= A,B$. We need to show $\tilde{Q}_i^A\leq\tilde{Q}_i^B$ for all $i\geq 1$.

For ensemble $\Pi$ let us define $I_{\Pi}(c):=\max\{i: Q^{\Pi}_i\geq N-c+1\}$, $c=1,\ldots,N$, $\Pi= A,B$. Define $I_{\Pi}(c)$ to be 0 if $Q_1^{\Pi}$ is (and hence all the $Q^{\Pi}_i$ values are) less than $N-c+1$. Note that $I_A(c)\leq I_B(c)$ for all $c= 1,2,\ldots N$ because of the initial ordering.

Now if the rule produces a removal of an item,  then the updated ensemble will have the values
\begin{equation}
\tilde{Q}^{\Pi}_i=
\begin{cases}
Q^{\Pi}_i-1, &\mbox{ for }i=I_{\Pi}(k),\\
Q^{\Pi}_i,&\mbox{ otherwise, }
\end{cases}
\end{equation}
if $I_{\Pi}(k)\geq 1$; otherwise all the $Q^{\Pi}_i$ values remain unchanged.
\begin{figure}
\begin{center}
\begin{tikzpicture}[scale=.5]
\draw (1,6)--(1,0)--(11,0)--(11,6);
\foreach \x in {10, 9,...,1}
	\draw (\x,1) rectangle (\x+1,0)
	(\x+.5,-.15) node [black,below] {\x} ;
\foreach \x in {10, 9,...,2}
	\draw (\x,2) rectangle (\x+1,1);
\foreach \x in {10, 9,...,4}
	\draw (\x,3) rectangle (\x+1,2);
\foreach \x in {10, 9,...,6}
	\draw (\x,4) rectangle (\x+1,3);
\foreach \x in {10, 9,...,9}
	\draw (\x,5) rectangle (\x+1,4);
\foreach \y in {1,2,...,5}
	\draw (11.15,\y-.5) node [black, right] {$Q_{\y}$};
\draw[thick] (7.5,-.67) circle [radius=.4];
\draw [fill=gray] (7,3) rectangle (8,4);

\draw (15,6)--(15,0)--(25,0)--(25,6);
\foreach \x in {10, 9,...,1}
	\draw (14+\x,1) rectangle (14+\x+1,0)
	(14+\x+.5,-.15) node [black,below] {\x} ;
\foreach \x in {24, 23,...,16}
	\draw (\x,2) rectangle (\x+1,1);
\foreach \x in {24, 23,...,18}
	\draw (\x,3) rectangle (\x+1,2);
\foreach \x in {24, 23,...,21}
	\draw (\x,4) rectangle (\x+1,3);
\foreach \x in {24, 23}
	\draw (\x,5) rectangle (\x+1,4);
\draw[dashed,thin, red] (20,3) rectangle (21,4);
\foreach \y in {1,2,...,5}
	\draw (25.15,\y-.5) node [black, right] {$Q_{\y}$};
\end{tikzpicture}
\caption{Removal of an item from the ensemble}\label{fig:removal}
\end{center}
\end{figure}
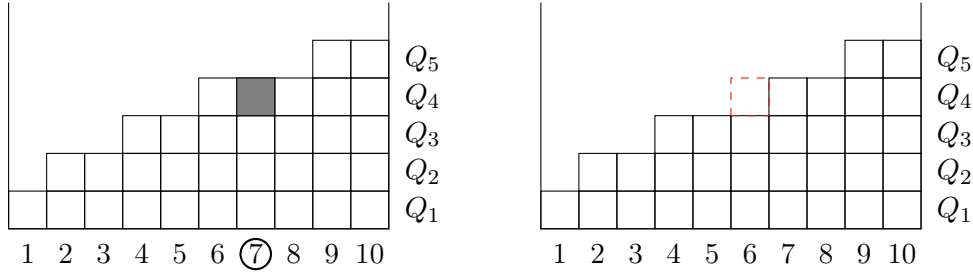
For example, in Figure~\ref{fig:removal}, $b=5$, $N=10$, and at the time of removal $k=7$. For this configuration $I_{\Pi}(7)=4$ since $Q^{\Pi}_4=5\geq 10-7+1=4$ but $Q^{\Pi}_5=2<4$. Hence, $Q_4^{\Pi}$ is reduced and all the other values remain unchanged. Note that the specific label of the servers does not matter here. So after the removal/addition of an item we consider the configuration as a whole by rearranging it again in non-decreasing order of the queue lengths.

Since in both $A$ and $B$ the values of $Q_i$ remain unchanged except for $i=I_A(k)$ and $I_B(k)$, it suffices to prove the preservation of the ordering for these two specific values of $i$. Now for $i=I_A(k)$, 
$$\tilde{Q}_i^A=Q_i^A-1\leq Q_i^B-1\leq\tilde{Q}_i^B.$$
If $I_B(k)=I_A(k)$, then we are done by the previous step. If $I_B(k)>I_A(k)$, then from the definition of $I_A(k)$ observe that $I_B(k)\notin\{i: Q_i^A\geq N-k+1\}$ and hence $Q_i^A<N-k+1$, for $i=I_B(k)$. Therefore, for $i=I_B(k)$,
$$\tilde{Q}_i^A\leq N-k\leq Q_i^B-1=\tilde{Q}_i^B.$$
On the other hand, if the rule produces the addition of an item to stack $l$, then the values will be updated as
\begin{equation}
\tilde{Q}^{\Pi}_i=
\begin{cases}
Q^{\Pi}_i+1, &\mbox{ for }i=I_{\Pi}(l)+1,\\
Q^{\Pi}_i,&\mbox{ otherwise, }
\end{cases}
\end{equation}
if $I_{\Pi}(l)<b_{\Pi}$, with $b_{\Pi}$ the stack-capacity of the corresponding system; otherwise the values remain unchanged.
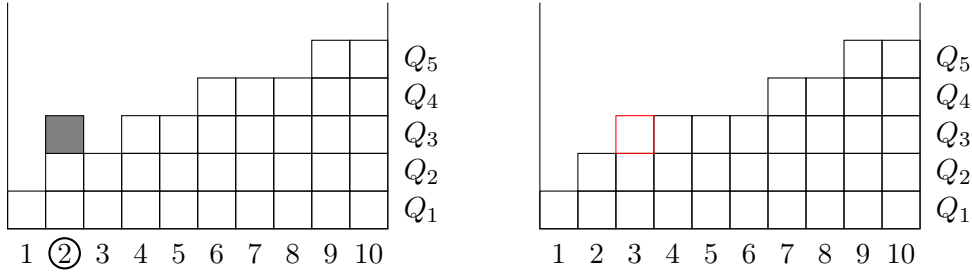
\begin{figure}
\begin{center}
\begin{tikzpicture}[scale=.5]
\draw (1,6)--(1,0)--(11,0)--(11,6);
\foreach \x in {10, 9,...,1}
	\draw (\x,1) rectangle (\x+1,0)
	(\x+.5,-.15) node [black,below] {\x} ;
\foreach \x in {10, 9,...,2}
	\draw (\x,2) rectangle (\x+1,1);
\foreach \x in {10, 9,...,4}
	\draw (\x,3) rectangle (\x+1,2);
\foreach \x in {10, 9,...,6}
	\draw (\x,4) rectangle (\x+1,3);
\foreach \x in {10, 9,...,9}
	\draw (\x,5) rectangle (\x+1,4);
\foreach \y in {1,2,...,5}
	\draw (11.15,\y-.5) node [black, right] {$Q_{\y}$};
\draw[thick] (2.5,-.65) circle [radius=.4];
\draw [fill=gray] (2,2) rectangle (3,3);

\draw (15,6)--(15,0)--(25,0)--(25,6);
\foreach \x in {10, 9,...,1}
	\draw (14+\x,1) rectangle (14+\x+1,0)
	(14+\x+.5,-.15) node [black,below] {\x} ;
\foreach \x in {24, 23,...,16}
	\draw (\x,2) rectangle (\x+1,1);
\foreach \x in {24, 23,...,18}
	\draw (\x,3) rectangle (\x+1,2);
\foreach \x in {24, 23,...,21}
	\draw (\x,4) rectangle (\x+1,3);
\foreach \x in {24, 23}
	\draw (\x,5) rectangle (\x+1,4);
\draw[red,thin] (17,2) rectangle (18,3);
\foreach \y in {1,2,...,5}
	\draw (25.15,\y-.5) node [black, right] {$Q_{\y}$};
\end{tikzpicture}
\caption{Addition of an item to the ensemble}\label{fig:addition}
\end{center}
\end{figure}
In Figure~\ref{fig:addition}, we have $l=2$ and for that particular configuration $I_{\Pi}(2)=2$. Hence, $Q_3^{\Pi}$ is incremented by one and the other variables remain fixed. 
\\Therefore, it is enough to consider the $i^{th}$ horizontal bars for $i=(I_A(l)+1), (I_B(l)+1)$ when $I_A(l)<b$. According to the addition rule there are several cases which we now consider one by one:
\begin{enumerate}
\item First we consider the case when in both ensembles the minimum stack height is less than $b-1$. Then by part (ii) of the rule both incoming items are added to the $l^{th}$   stack. When considering ensemble $B$ we may neglect the case $I_B(l)\geq b$ since then the value at $I_B(l)+1$ does not matter. Thus assume $I_B(l)\leq b-1$ and set $i=I_B(l)+1$ so that
$$\tilde{Q}_i^B=Q_i^B+1\geq Q_i^A+1\geq\tilde{Q}_i^A.$$
If $I_A(l)=I_B(l)$, then we are done by the previous case. If $I_A(l)+1\leq I_B(l)$, then it follows from the definition that $Q_i^A<N-l+1$ and $Q_i^B\geq N-l+1$, for $i=I_A(l)+1$. Hence,
$$\tilde{Q}_i^A=Q_i^A+1\leq N-l+1\leq Q_i^B\leq\tilde{Q}_i^B.$$
\item If the minimum stack height in $A$ is less than $b-1$ and that in $B$ is precisely $b-1$, then according to the rule the incoming item is added to the $l^{th}$   stack in $A$ and the $l_B^{th}$   stack in $B$. We here show that the component-wise ordering will be preserved if either $l=1$ or $l\geq l_B$. Observe that if $l=1$, then $I_A(l)<b-1$ which implies $I_A(l)+1\leq b-1$. But since the minimum stack height in $B$ is $b-1$, for all $i\leq b-1$ and in particular for $i=I_A(l)+1$, $\tilde{Q}^B_i=N\geq\tilde{Q}^A_i$. Now we consider the case when $l\geq l_B$. Also observe that the fact that the minimum stack height in $B$ is $b-1$, implies $I_B(l_B)\geq b-1\geq I_A(l_A)$ (since if $I_A(l)=b$, then nothing will be changed and so we do not need to consider this case). Then again if $I_A(l)=I_B(l_B)$, we are done.  Therefore, suppose $I_A(l)<I_B(l_B)$, which implies $I_A(l)+1\leq I_B(l_B)$. By definition, for $i=I_A(l)+1$, we have $Q_i^A<N-l+1$ and $Q_i^B\geq N-l_B+1\geq N-l+1$. Combining these two inequalities yields
$$\tilde{Q}_i^A=Q_i^A+1\leq N-l+1\leq Q_i^B=\tilde{Q}_i^B.$$
\item If the minimum stack height in both ensembles is $b-1$, then recall that the incoming item is added to the $l_A^{th}$   stack in $A$ and to the $l_B^{th}$   stack in $B$ with $l_A\geq l_B$. Arguing similarly as in the previous case we can conclude that the inequality is preserved.
\item Finally, if the minimum stack height in $B$ is larger than or equal to $b$, then the preservation of the inequality is trivial.
\end{enumerate}
Hence, the proof of the proposition is complete.
\end{proof}
\subsection{The coupling construction}\label{subsec: coupling}
We now construct a coupling between two systems $A$ and $B$ following any two schemes, say, $\Pi_A=\Pi(l_0,l_1,\ldots,l_{b-1})$ and $\Pi_B=\Pi(d_0, d_1,\ldots, d_{b'-1})$ in $\Pi^{(N)}$ respectively and combine it with Proposition~\ref{prop: det_ord} to get the desired stochastic ordering results.

For the arrival process we couple the two systems as follows. First we synchronize the  arrival epochs of the two systems. Now assume that in the systems $A$ and $B$, the minimum queue lengths are $k$ and $m$, respectively, $k\leq b-1$, $m\leq b'-1$. Therefore, when a task arrives, the dispatchers in $A$ and $B$ have to select $l_k$ and $d_m$ servers, respectively, and then have to send the task to the one having the minimum queue length among the respectively selected servers. Since the servers are being selected uniformly at random we can assume without loss of generality, as in the stack construction, that the servers are arranged in non-decreasing order of their queue lengths and are indexed in increasing order. 
Hence, observe that when a few server indices are selected, the server having the minimum of those indices will be the server with the minimum queue length among these. Hence, in this case the dispatchers in $A$ and $B$ select $l_k$ and $d_m$ random numbers (without replacement) from $\{1,2,\ldots,N\}$ and then send the incoming task to the servers having indices to be the minimum of those selected numbers. To couple the decisions of the two systems, at each arrival epoch a single random permutation of $\{1,2,\ldots,N\}$ is drawn, denoted by $\mathbf{\Sigma}^{(N)}:=(\sigma_1, \sigma_2,\ldots,\sigma_N)$. Define $\sigma_{(i)}:= \min_{j\leq i}\sigma_j$. Then observe that system $A$ sends the task to the server with the index $\sigma_{(l_k)}$ and system $B$ sends the task to the server with the index $\sigma_{(d_m)}$. Since at each arrival epoch both systems use a common random permutation, they take decisions in a coupled manner.

For the potential departure process, couple the service completion times of the $k^{th}$ queue in both scenarios, $k= 1,2,\ldots,N$. More precisely, for the potential departure process assume that we have a single synchronized exp($N$) clock independent of arrival epochs for both systems. Now when this clock rings, a number $k$ is uniformly selected from $\{1,2,\ldots,N\}$ and a potential departure occurs from the $k^{th}$ queue in both systems. If at a potential departure epoch an empty queue is selected, then we do nothing. In this way the two schemes, considered independently, still evolve according to their appropriate statistical laws.

Loosely speaking, our next result is based upon the following intuition: Suppose we have two systems $A$ and $B$ with two different schemes $\Pi_A$ and $\Pi_B$ having buffer sizes $b$ and $b'$ ($b\leq b'$) respectively. 
Also, for these two systems, initially, $Q^A_i\leq Q^B_i$ for all $i=1,\ldots,b$. Below we develop some intuition as to under what conditions the initial ordering of the $Q_i$-values will be preserved after one arrival or departure.

For the departure process if we ensure that departures will occur from the $k^{th}$ largest queue in both systems for some $k\in\{1,2,\ldots,N\}$ (ties are broken in any way), then observe that the ordering will be preserved after one departure.

In case of the arrival process, assume that when the minimum queue length in both systems is less than $b-1$, the incoming task is sent to the server with the same index. In that case it can be seen that the $Q_i$-values in $A$ and $B$ will preserve their ordering after the arrival as well. Next consider the case when the minimum queue length in both systems is precisely $b-1$. Now, in $A$, an incoming task can either be rejected (and will not change the $Q$-values at all) or be accepted (and $Q^{\Pi_A}_b$ will increase by 1). Here we ensure that if the incoming task is accepted in $A$, then it is accepted in $B$ as well unless $Q_b^{\Pi_A}<Q_b^{\Pi_B}$, in which case it is clear that the initial ordering will be preserved after the arrival. Finally, if the minimum queue length in $A$ is less than $b-1$ and that in $B$ is precisely $b-1$, then the way to ensure the inequality is either by making the scheme $\Pi_A$ send the incoming task to the server with minimum queue length (and hence, it will only increase the value of $Q_i^{\Pi_A}$ for some $i<b$, leaving other values unchanged) or by letting the selected server in $\Pi_A$ have a smaller queue length than the selected server in $\Pi_B$. The former case corresponds to the condition $d=N$ and the latter corresponds to the condition $d\leq d_{b-1}$, either of which has to be satisfied, in order to ensure the preservation of the ordering. This whole idea is formalized below.

\begin{proposition}\label{prop: stoch_ord}
For two schemes $\Pi_A=\Pi(l_0,l_1,\ldots,l_{b-1})$ and $\Pi_B=\Pi(d_0, d_1,\ldots, d_{b'-1})$ with $b\leq b'$ assume $l_0=\ldots=l_{b-2}=d_0=\ldots=d_{b-2}=d$, $l_{b-1}\leq d_{b-1}$ and either $d=N$ or $d\leq d_{b-1}$. Then the following holds:
\begin{enumerate}[{\normalfont (i)}]
\item\label{component_ordering} $\{Q^{ \Pi_A}_i(t)\}_{t\geq 0}\leq_{st}\{Q^{ \Pi_B}_i(t)\}_{t\geq 0}$ for $i=1,2,\ldots,b$
\item\label{upper bound} $\{\sum_{i=1}^b Q^{ \Pi_A}_i(t)+L^{ \Pi_A}(t)\}_{t\geq 0}\geq_{st} \{\sum_{i=1}^{b'} Q^{ \Pi_B}_i(t)+L^{ \Pi_B}(t)\}_{t\geq 0}$
\item\label{delta_ineq} $\{\Delta(t)\}_{t\geq 0}\geq \{\sum_{i=b+1}^{b'}Q_i^{ \Pi_B}(t)\}_{t\geq 0}$ almost surely under the coupling defined above,
\end{enumerate}
for any fixed $N\in\mathbbm{N}$ where $\Delta(t):=L^{ \Pi_A}(t)-L^{ \Pi_B}(t)$, provided that at time $t=0$ the above ordering holds.
\end{proposition}
\begin{proof}
To prove the stochastic ordering we use the coupling of the schemes as described above and show that the ordering holds for the entire sample path. That is, the two processes arising from the above pair of schemes will be defined on a common probability space and it will then be shown that the ordering is maintained almost surely over all time.

Note that we shall consider only the event times $0=t_0<t_1<\ldots$, \emph{i.e.}~the time epochs when arrivals or potential service completions occur and apply forward induction to show that the ordering is preserved. By assumption the orderings hold at time $t_0 = 0$.\\

(i) The main idea of the proof is to use the coupling and show that at each event time the joint process of the two schemes follows a rule $Rule(k,l,l_A,l_B)$ described in Subsection~\ref{subsec: det_ord}, with some random $k$, $l$, $l_A$ and $l_B$ such that $l_A\geq l_B$ and either $l=1$ or $l\geq l_B$, and apply Proposition~\ref{prop: det_ord}. 
We now identify the rule at event time $t_1$ and verify that the conditions of Proposition~\ref{prop: det_ord} hold. If the event time $t_1$ is a potential departure epoch, then according to the coupling similarly as in the stack formation a random $k\in\{1, 2,\ldots, N\}$ will be chosen in both systems for a potential departure. Now assume that $t_1$ is an arrival epoch. 
In that case if the minimum queue length in both systems is less than $b-1$, then both schemes $ \Pi_A$ and $ \Pi_B$ will send the arriving task to the $\sigma_{(d)}^{th}$   queue. If the minimum queue length in scheme $ \Pi_A$ is $b-1$, then the incoming task is sent to the $\sigma_{(l_{b-1})}^{th}$ queue and if in scheme $ \Pi_B$ the minimum queue length is $b-1$, then the incoming task is sent to $\sigma_{(d_{b-1})}^{th}$ queue where we recall that  $(\sigma_1, \sigma_2,\ldots,\sigma_N)$ is a random permutation of $\{1,2,\ldots,N\}$. Therefore, observe that at each step $Rule(\sigma_{(d)},k,\sigma_{(l_{b-1})}, \sigma_{(d_{b-1})})$ is followed.

Now to check the conditions, first observe that 
$$\sigma_{(l_{b-1})}= \min_{i\leq l_{b-1}}\sigma_i\geq\min_{i\leq d_{b-1}}\sigma_i=\sigma_{(d_{b-1})},$$
where the second inequality is due to the assumption $l_{b-1}\leq d_{b-1}$. In addition, we have assumed either $d=N$ or $d\leq d_{b-1}$. If $d=N$, then the dispatcher sends the incoming task to the server with the minimum queue length which is the same as sending to stack 1 as in Proposition $\ref{prop: det_ord}$. On the other hand, $d\leq d_{b-1}$ implies
$$\sigma_{(d)}= \min_{i\leq d}\sigma_i\geq \min_{i\leq d_{b-1}}\sigma_i=\sigma_{(d_{b-1})}.$$
Therefore, assertion~\eqref{component_ordering} follows from Proposition~\ref{prop: det_ord}.\\

(ii) We again apply forward induction. Assume that the ordering holds at time $t_0$. If the next event time is an arrival epoch, then observe that both sides of the inequality in \eqref{upper bound} will increase, since if the incoming task is accepted, then the $Q$-values will increase and if it is rejected, then the $L$-value will increase.\\
On the other hand, if the next event time is a potential departure epoch, then it suffices to show that,  if the left-hand-side decreases, then the right-hand-side decreases as well. Indeed, from assertion~\eqref{component_ordering} we know that $Q^{ \Pi_A}_1\leq Q^{ \Pi_B}_1$ and hence we can see that if there is a departure from $ \Pi_A$ (\emph{i.e.}~the $k^{th}$ queue of $\Pi_A$ is non-empty), then there will be a departure from $ \Pi_B$ (\emph{i.e.}~the $k^{th}$ queue of $\Pi_B$ will be non-empty) as well.\\

(iii) Assertion~\eqref{delta_ineq} follows directly from~\eqref{component_ordering} and  \eqref{upper bound}.
\end{proof}

\subsection{Discussion}
It is worth emphasizing that Proposition~\ref{prop: stoch_ord}\eqref{component_ordering} is fundamentally different from the stochastic majorization results
for the ordinary JSQ policy, and below we contrast our methodology with some existing literature.
As noted earlier, the ensemble of stacks, arranged in non-decreasing
order, represents the empirical CDF of the queue length distribution
at the various servers.
Specifically, if we randomly select one of the servers, then the
probability that the queue length at that server is greater than
or equal to~$i$ at time~$t$ under policy~$\Pi$ equals
$\frac{1}{N} \mathbbm{E} Q_i^\Pi(t)$.
Thus assertion~\eqref{component_ordering} of Proposition~\ref{prop: stoch_ord} implies that if we select one
of the servers at random, then its queue length is stochastically
larger under policy~$\Pi_B$ than under policy~$\Pi_A$.

The latter property does generally \emph{not} hold when we compare
the ordinary JSQ policy with an alternative load balancing policy.
Indeed,  the class of load balancing schemes $\tilde{\Pi}^{(N)}$ (for the $N^{th}$ system say) considered in \cite{Towsley1992} consists of all the schemes that have instantaneous queue length information of all the servers and that have to send an incoming task to some server if there is at least some place available anywhere in the whole system. This means that a scheme can only discard an incoming task if the system is completely full. Observe that \emph{only} the JSQ policy lies both in the class $\Pi^{(N)}$ (defined in Section~\ref{sec: coupling}) and the class $\tilde{\Pi}^{(N)}$, because any scheme in $\Pi^{(N)}$ other than JSQ may reject an incoming task in some situations, where there might be some place available in the system. In this setup \cite{Towsley1992} shows that for any scheme $\Pi\in\tilde{\Pi}^{(N)}$, and for all $t\geq 0$,
\begin{align}\label{eq: towsley}
\sum_{i=1}^k Y_{(i)}^{JSQ}(t)&\leq_{st}\sum_{i=1}^k Y_{(i)}^{\Pi}(t),\mbox{ for } k=1,2,\ldots, N,\\
\{L^{JSQ}(t)\}_{t\geq 0}&\leq_{st}\{L^{\Pi}(t)\}_{t\geq 0},
\end{align}
where $Y^{\Pi}_{(i)}(t)$ is the $i^{th}$ largest queue length at time $t$ in the system following scheme $\Pi$ and $L^{\Pi}(t)$ is the total number of overflow events under policy $\Pi$ up to time $t$, as defined in Section~\ref{sec: model descr}. Observe that $Y_{(i)}^{\Pi}$ can be visualized as the $i^{th}$ largest vertical bar (or stack) as described in Subsection~\ref{subsec: det_ord}. Thus~\eqref{eq: towsley} says that the sum of the lengths of the $k$ largest \emph{vertical} stacks in a system following any scheme $\Pi\in\tilde{\Pi}^{(N)}$ is stochastically larger than or equal to that following the scheme JSQ for any $k=1,2,\ldots,N$. Mathematically, this ordering can be written as
$$\sum_{i = 1}^{b} \min\{k, Q_i^{JSQ}(t)\}  \leq_{st}
\sum_{i = 1}^{b} \min\{k, Q_i^{\Pi}(t)\},$$
for all $k = 1, \dots, N$.
In contrast, Proposition~\ref{prop: stoch_ord} shows that the length of the $i^{th}$ largest \emph{horizontal} bar in the system following some scheme $\Pi_A$ is stochastically smaller than that following some other scheme $\Pi_B$ if some conditions are satisfied. Also observe that the ordering between each of the horizontal bars (\emph{i.e.}~$Q_i$'s) implies the ordering between the sums of the $k$ largest vertical stacks, but not the other way around. Further it should be stressed that, in crude terms, JSQ in our class $\Pi^{(N)}$, plays the role of upper bound, whereas what Equation~\eqref{eq: towsley} implies is almost the opposite in nature to the conditions we require.

While in \cite{Towsley1992} no policies with admission control (where the dispatcher can discard an incoming task even if the system is not full) were considered, in a later paper \cite{towsley} and also in \cite{Towsley95} the class was extended to a class $\hat{\Pi}^{(N)}$ consisting of all the policies that have information about instantaneous queue lengths available and that can either send an incoming task to some server with available space or can reject an incoming task even if the system is not full. One can see that $\hat{\Pi}^{(N)}$ contains both $\tilde{\Pi}^{(N)}$ and $\Pi^{(N)}$ as subclasses. But then for such a class with admission control, \cite{towsley} notes that a stochastic ordering result like \eqref{eq: towsley} cannot possibly hold. Instead, what was shown in \cite{Towsley95} is that for all $t\geq 0$,
\begin{align}\label{eq:ordering total jobs}
\sum_{i=1}^{k}Y_{(i)}^{JSQ}(t)+L^{JSQ}(t)\leq_{st}\sum_{i=1}^k Y_{(i)}^{\Pi}(t)+L^{\Pi}(t)\mbox{ for all }k\in\{1,2,\ldots,N\}
\end{align}
Note that the ordering in \eqref{eq:ordering total jobs} is the same in spirit as the ordering in Proposition~\ref{prop: stoch_ord}\eqref{upper bound} and the inequalities in \eqref{eq:ordering total jobs} are in the language of \cite[Def.~14.4]{Towsley95}, \emph{weak sub-majorization by $p$}, where $p=L^{\Pi}(t)-L^{JSQ}(t)$. But in this case also our inequalities in Proposition~\ref{prop: stoch_ord}\eqref{component_ordering} imply something completely orthogonal to what is implied by \eqref{eq:ordering total jobs}. In other words, the stochastic ordering results in Proposition~\ref{prop: stoch_ord} provide both upper and lower bounds for the occupancy state of one scheme w.r.t another and 
are stronger than the stochastic majorization properties for the JSQ
policy existing in the literature. Hence we also needed to exploit a different proof methodology
than the majorization framework developed in~\cite{towsley, Towsley95, Towsley1992}.

\section{Convergence on diffusion scale}\label{sec: conv}
In this section we leverage the stochastic ordering established in Proposition~\ref{prop: stoch_ord} to prove the main result stated in Theorem~\ref{th: main}. All the inequalities below are stated as  almost sure statements with respect to the common probability space constructed under the associated coupling. We shall use this joint probability space to make the probability statements about the marginals.
\begin{proof}[Proof of Theorem~\ref{th: main}]
Let $\Pi=\Pi(N,d_1,\ldots,d_{b-1})$ be a load balancing scheme in the class $\Pi^{(N)}$. Denote by $\Pi_1$ the scheme $\Pi(N,d_1)$ with buffer size $b=2$ and let $\Pi_2$ denote the JIQ policy $\Pi(N,1)$ with buffer size $b=2$.

Observe that from Proposition~\ref{prop: stoch_ord} we have under the coupling defined in Subsection~\ref{subsec: coupling},
\begin{equation}\label{eq: bound}
\begin{split}
|Q^{\Pi}_i(t)-Q^{\Pi_2}_i(t)|&\leq |Q^{\Pi}_i(t)-Q^{\Pi_1}_i(t)|+|Q^{\Pi_1}_i(t)-Q^{\Pi_2}_i(t)|\\
&\leq |L^{\Pi_1}(t)-L^{\Pi}(t)|+|L^{\Pi_2}(t)-L^{\Pi_1}(t)|\\
&\leq 2L^{\Pi_2}(t),
\end{split}
\end{equation}
for all $i\geq 1$ and $t\geq 0$ with the understanding that $Q_j(t)=0$ for all $j>b$, for a scheme with buffer $b$. The third inequality above is due to Proposition~\ref{prop: stoch_ord}\eqref{delta_ineq}, which in particular says that $\{L^{\Pi_2}(t)\}_{t\geq 0}\geq \{L^{\Pi_1}(t)\}_{t\geq 0}\geq\{ L^{\Pi}(t)\}_{t\geq 0}$ almost surely under the coupling. Now we have the following lemma which we will prove below.

\begin{lemma}\label{lem: tight}
For all $t\geq 0$, under the assumption of Theorem~\ref{th: main}, $\{L^{\Pi_2}(t)\}_{N\geq 1}$ forms a tight sequence.
\end{lemma}

Since $L^{\Pi_2}(t)$ is non-decreasing in $t$, the above lemma in particular implies that
\begin{equation}\label{eq: conv 0}
\sup_{t\in[0,T]}\frac{L^{\Pi_2}(t)}{\sqrt{N}}\prob 0.
\end{equation}
For any scheme $\Pi\in\Pi^{(N)}$, from \eqref{eq: bound} we know that 
$$\{Q^{\Pi_2}_i(t)-2L^{\Pi_2}(t)\}_{t\geq 0}\leq\{Q^{\Pi}_i(t)\}_{t\geq 0}\leq\{Q^{\Pi_2}_i(t)+2L^{\Pi_2}(t)\}_{t\geq 0}.$$
Combining \eqref{eq: bound} and \eqref{eq: conv 0} shows that if the weak limits under the $\sqrt{N}$ scaling exist with respect to the Skorohod $J_1$-topology, they must be the same for all the schemes in the class $\Pi^{(N)}$. Also from Theorem 2 in \cite{EG15} we know that the weak limit for $\Pi(N,N)$ exists and the common weak limit for the first two components can be described by the unique solution in $D\times D$ of the stochastic differential equations in \eqref{eq: main theorem}. Hence the proof of Theorem~\ref{th: main} is complete.
\end{proof}

\begin{proof}[Proof of Lemma~\ref{lem: tight}]
First we consider the evolution of $L^{\Pi_2}(t)$ as the following unit jump counting process. A task arrival occurs at rate $\lambda_N$ at the dispatcher, and if $Q_1^{\Pi_1}=N$, then it sends it to a server chosen uniformly at random. If the chosen server has queue length 2, then $L^{\Pi_2}$ is increased by 1. It is easy to observe that this evolution can be equivalently described as follows. If $Q^{\Pi_2}_1(t)=N$, then each of the servers having queue length 2 starts increasing $L^{\Pi_2}$ by 1 at rate $\lambda_N/N$. From this description we have
\begin{equation}\label{eq:L_rep}
L^{\Pi_2}(t)=A\left(\int_0^t\frac{\lambda_N}{N}Q^{\Pi_2}_2(s)\mathbbm{1}[Q^{\Pi_2}_1(s)=N]ds\right)
\end{equation}
with $A(\cdot)$ being a unit rate Poisson process. Now using Proposition~\ref{prop: stoch_ord} it follows that $\mathbbm{1}[Q^{\Pi_2}_1(s)=N]\leq \mathbbm{1}[Q^{\Pi_3}_1(s)=N]$ and $Q^{\Pi_2}_2(s)\leq Q^{\Pi_3}_2(s)$ where $\Pi_3=\Pi(N,N)$. Therefore, it is enough to prove the stochastic boundedness \cite[Def.~5.4]{PTRW07} of the sequence 
\begin{equation}
\Gamma^{(N)}(t):=A\left(\int_0^t\frac{\lambda_N}{N}Q^{\Pi_3}_2(s)\mathbbm{1}[Q^{\Pi_3}_1(s)=N]ds\right).
\end{equation}
To prove this we shall use the martingale techniques described for instance in \cite{PTRW07}. 
Define the filtration $\mathbf{F}\equiv\{\mathcal{F}_t:t\geq 0\}$, where  for $t\geq 0$,
$$\mathcal{F}_t:=\sigma\left(Q^{\Pi_3}(0),A\left(\int_0^t\frac{\lambda_N}{N}Q^{\Pi_3}_2(s)\mathbbm{1}[Q^{\Pi_3}_1(s)=N]ds\right), Q^{\Pi_3}_1(s), Q^{\Pi_3}_2(s): 0\leq s\leq t\right).$$
Then using a random time change of unit rate Poisson process \cite[Lemma 3.2]{PTRW07} and similar arguments to those in \cite[Lemma 3.4]{PTRW07}, we have the next lemma.
\begin{lemma}
With respect to the filtration $\mathbf{F}$,
\begin{equation}
M^{(N)}(t):= A\left(\int_0^t\frac{\lambda_N}{N}Q^{\Pi_3}_2(s)\mathbbm{1}[Q^{\Pi_3}_1(s)=N]ds\right)-\int_0^t\frac{\lambda_N}{N}Q^{\Pi_3}_2(s)\mathbbm{1}[Q^{\Pi_3}_1(s)=N]ds
\end{equation}
is a square-integrable martingale with $\mathbf{F}$-compensator $$I(t)=\int_0^t\frac{\lambda_N}{N}Q^{\Pi_3}_2(s)\mathbbm{1}[Q^{\Pi_3}_1(s)=N]ds.$$ Moreover, the predictable quadratic variation process is given by $\langle M^{(N)}\rangle(t)=I(t).$
\end{lemma}
Now we apply Lemma 5.8 in \cite{PTRW07} which gives a stochastic boundedness criterion for square-integrable martingales. 
\begin{lemma}
\begin{normalfont}
\cite[Lemma 5.8]{PTRW07}
\end{normalfont}
Suppose that, for each $N\geq 1$, $M^{(N)}\equiv \{M^{(N)}(t):t\geq 0\}$ is a square-integrable martingale (with respect to a specified filtration) with predictable quadratic variation process $\langle M^{(N)}\rangle\equiv\{\langle M^{(N)}\rangle(t):t\geq 0\}$. If the sequence of random variables $\{\langle M^{(N)}\rangle(T): N\geq 1\}$ is stochastically bounded in $\mathbbm{R}$ for each $T>0$, then the sequence of stochastic processes $\{M^{(N)}: N\geq 1\}$ is stochastically bounded in $D$.
\end{lemma}
Therefore, it only remains to show the stochastic boundedness of $\{\langle M^{(N)}\rangle(T):N\geq 1\}$ for each $T>0$. Fix a $T>0$ and observe that
\begin{equation}\label{eq: qvp_boundedness}
\begin{split}
\langle M^{(N)}\rangle(T)&=\frac{\lambda_N}{N}\int_0^T\frac{Q^{\Pi_3}_2(s)}{\sqrt{N}}\mathbbm{1}[Q^{\Pi_3}_1(s)=N]ds\\
&\leq \left[\sup_{t\in [0,T]}\frac{Q^{\Pi_3}_2(s)}{\sqrt{N}}\right]\times\left[\int_0^{T}\frac{1}{\sqrt{N}}\mathbbm{1}[Q^{\Pi_3}_1(s)=N]\lambda_N ds\right].
\end{split}
\end{equation}
From \cite{EG15} we know that $\sup_{t\in [0,T]}Q^{\Pi_3}_2(t)/\sqrt{N}$ and $\int_0^{T}1/\sqrt{N}\mathbbm{1}[Q^{\Pi_3}_1(s)=N]dA(\lambda_Ns)$  are both tight. Moreover, since $\int_0^{T}1/\sqrt{N}\mathbbm{1}[Q^{\Pi_3}_1(s)=N]\lambda_N ds$ is the intensity function of the stochastic integral $\int_0^{T}1/\sqrt{N}\mathbbm{1}[Q^{\Pi_3}_1(s)=N]dA(\lambda_Ns)$, which is a tight sequence, we have the following lemma.
\begin{lemma}
For all fixed $T\geq 0$, 
$\int_0^{T}\frac{1}{\sqrt{N}}\mathbbm{1}[Q^{\Pi_3}_1(s)=N]\lambda_N ds$ is tight as a sequence in $N$.
\end{lemma}
Hence, both terms on the right-hand side of \eqref{eq: qvp_boundedness} are stochastically bounded and the resulting stochastic bound on $\langle M^{(N)}\rangle(T)$ completes the proof.
\end{proof}

\section{Conclusion}\label{sec:conclusion}
In the present paper we have considered a system with symmetric
Markovian parallel queues and a single dispatcher.
We established the diffusion limit of the queue process in the
Halfin-Whitt regime for a wide class of load balancing schemes which
always assign an incoming task to an idle server, if there is any.
The results imply that assigning tasks to idle servers whenever
possible is sufficient to achieve diffusion level optimality.
Thus, using more fine-grained queue state information will increase the
communication burden and potentially impact the scalability in
large-scale deployments without significantly improving the performance.

In ongoing work we are aiming to extend the analysis to the stationary
distribution of the queue process, and in particular to quantify the
performance deviation from a system with a single centralized queue.
It would also be interesting to generalize the results to scenarios
where the individual nodes have general state-dependent service
rates rather than constant service rates.

\section*{Acknowledgments}
This research was financially supported by an ERC Starting Grant and by The Netherlands Organization for Scientific Research (NWO) through TOP-GO grant 613.001.012 and Gravitation Networks grant 024.002.003. Dr. Whiting was supported in part by an Australian Research grant DP-1592400 and in part by a Macquarie University Vice-Chancellor Innovation Fellowship.

\bibliographystyle{apa}

\bibliography{bib_jiq}
\end{document}